\pdfoutput=1 

\documentclass[11pt]{amsart}

\usepackage{amssymb,amsthm,enumitem,colonequals,tikz-cd,microtype}
\usepackage[normalem]{ulem} 

\usepackage[osf]{Baskervaldx}
\usepackage[baskervaldx]{newtxmath}
\usepackage[cal=boondoxo]{mathalfa} 

\usepackage[top=3.75cm, bottom=3cm, left=3.5cm, right=3.5cm]{geometry}

\usepackage{xcolor}
\colorlet{darkblue}{blue!55!black}
\colorlet{darkcyan}{cyan!50!black}
\colorlet{darkgreen}{green!60!black}

\PassOptionsToPackage{hyphens}{url}
\usepackage{hyperref}
\hypersetup{
    colorlinks=true,
    linkcolor=darkblue,
    urlcolor=darkcyan,
    citecolor=darkgreen,
}

\def\eqref#1{\textcolor{darkblue}{(\ref{#1})}}

\usepackage[nameinlink]{cleveref} 
\Crefformat{section}{#2\S#1#3}
\Crefmultiformat{section}{#2\S\S#1#3}{ and~#2#1#3}{, #2#1#3}{, and~#2#1#3}
\Crefname{equation}{Diagram}{Diagrams}
\crefname{equation}{Equation}{Equations}

\usepackage[pagewise]{lineno}
\let\oldequation\equation
\let\oldendequation\endequation
\renewenvironment{equation}{\linenomathNonumbers\oldequation}{\oldendequation\endlinenomath}
\expandafter\let\expandafter\oldequationstar\csname equation*\endcsname
\expandafter\let\expandafter\oldendequationstar\csname endequation*\endcsname
\renewenvironment{equation*}{\linenomathNonumbers\oldequationstar}{\oldendequationstar\endlinenomath}
\let\oldalign\align
\let\oldendalign\endalign
\renewenvironment{align}{\linenomathNonumbers\oldalign}{\oldendalign\endlinenomath}
\expandafter\let\expandafter\oldalignstar\csname align*\endcsname
\expandafter\let\expandafter\oldendalignstar\csname endalign*\endcsname
\renewenvironment{align*}{\linenomathNonumbers\oldalignstar}{\oldendalignstar\endlinenomath}

\theoremstyle{plain}
\newtheorem{theorem}{Theorem}[section]
\newtheorem{lemma}[theorem]{Lemma}
\newtheorem{corollary}[theorem]{Corollary}
\newtheorem{proposition}[theorem]{Proposition}
\newtheorem*{theorem*}{Theorem}
\newtheorem*{corollary*}{Corollary}
\newtheorem*{proposition*}{Proposition}

\theoremstyle{definition}
\newtheorem{definition}[theorem]{Definition}
\newtheorem{example}[theorem]{Example}
\newtheorem{remark}[theorem]{Remark}

\newtheorem*{ack}{Acknowledgments}
\newtheorem*{warning}{Warning}

\AddToHook{env/conjecture/begin}{\crefalias{theorem}{conjecture}}
\AddToHook{env/lemma/begin}{\crefalias{theorem}{lemma}}
\AddToHook{env/corollary/begin}{\crefalias{theorem}{corollary}}
\AddToHook{env/proposition/begin}{\crefalias{theorem}{proposition}}
\AddToHook{env/definition/begin}{\crefalias{theorem}{definition}}
\AddToHook{env/remark/begin}{\crefalias{theorem}{remark}}
\AddToHook{env/setup/begin}{\crefalias{theorem}{setup}}
\AddToHook{env/example/begin}{\crefalias{theorem}{example}}
\AddToHook{env/conjecture/begin}{\crefalias{theorem}{Conjecture}}

\numberwithin{equation}{section}
\numberwithin{theorem}{section}

\title{A note on quasi-perfect morphisms}

\author[T.~De Deyn]{Timothy De Deyn}
\address{T.~De Deyn,
Max Planck Institute for Mathematics,
Bonn, Germany}
\email{dedeyn@mpim-bonn.mpg.de}

\author[P.~Lank]{Pat Lank}
\address{P.~Lank,
Dipartimento di Matematica “F. Enriques”, Universit\`{a} degli Studi di Milano, Via Cesare
Saldini 50, 20133 Milano, Italy}
\email{plankmathematics@gmail.com}

\author[K.~Manali Rahul]{Kabeer Manali Rahul}
\address{K.~Manali Rahul,
Max Planck Institute for Mathematics,
Bonn, Germany}
\email{kabeermr.maths@gmail.com}

\date{\today}



\begin{document}
    
\begin{abstract}
    This note is concerned with quasi-perfect morphisms between Noetherian algebraic spaces. In particular, we study the local behavior of quasi-perfect proper morphisms. We show that quasi-perfectness of a proper morphism can be detected at the \'{e}tale local rings of points of the target, as well as their completions and (strict) Henselizations. As a corollary, we obtain that the locus of points where a proper morphism is quasi-perfect is Zariski open.
\end{abstract}

\maketitle


\section{Introduction}
\label{sec:intro}

Given a quasi-compact and quasi-separated morphism $f\colon Y \to X$ of schemes, Neeman’s Brown Representability Theorem provides a right adjoint $f^\times$ to the derived pushforward $\mathbf{R}f_\ast$ on $D_{\operatorname{qc}}$, the derived category of complexes with quasi-coherent cohomology. While $f^\times$ preserves small coproducts for complexes that are uniformly homologically bounded below, this property fails in general. Motivated by \cite[Lemma 1 \& Corollary 2]{Verdier:1969}, Lipman--Neeman \cite{Lipman/Neeman:2007} introduced the notion of a \emph{quasi-perfect morphism} to remedy this.

The defining property is that $f^\times$ preserves small coproducts on all of $D_{\operatorname{qc}}$ which is equivalent to $\mathbf{R}f_\ast$ preserving perfect complexes; see \cite[Theorem 1.2 \& Proposition 2.1]{Lipman/Neeman:2007}. There are numerous examples of quasi-perfect morphisms, e.g.\ any proper morphisms of finite Tor dimension between Noetherian schemes. However, quasi-perfectness generally fails; for example, for closed immersions into nonregular Noetherian schemes.

More recent developments have recast these ideas in a categorical framework. For example, Balmer--Dell'Ambrogio--Sanders develop a duality theory for rigidly-compactly generated tensor-triangulated categories \cite[Theorem 1.7]{Balmer/Dellambrogio/Sanders:2016}. However, when interpreted in the setting of schemes, this perspective is largely global in nature. In particular, the local behavior of quasi-perfectness remains less understood; for instance, how quasi-perfectness at the local rings of the target influences quasi-perfectness globally.

In this note, we study the local behavior of quasi-perfectness for proper morphisms between algebraic spaces. 

\begin{theorem*}
    [\Cref{thm:quasi-perfectness_locality}]
    Let $f\colon Y\to X$ be a proper morphism of Noetherian algebraic spaces. Suppose $A_\# \in \{\mathcal{O}_{X,\#}^h, \mathcal{O}_{X,\#}^{sh} \}$. 
    Then the following are equivalent:
    \begin{enumerate}
       \item $f$ is quasi-perfect,
       \item the base change of $f$ to $\operatorname{Spec}(A_p)$ is quasi-perfect for all $p \in |X|$.
    \end{enumerate}
    Furthermore, if $s\colon U\to X$ is an \'{e}tale surjective morphism from a quasi-compact scheme, one can take $A_\# \in \{ \mathcal{O}_{U,\#}, \widehat{\mathcal{O}}_{U,\#}\}$
    and let $p$ run over all points in $U$.
\end{theorem*}

To the best of our knowledge, the observed behavior is new even for schemes. \Cref{thm:quasi-perfectness_locality} provides a mechanism for detecting quasi-perfectness by reducing the question to targets that are local rings. 
The proof relies on several auxiliary lemmas, potentially of independent interest, that allow for quasi-perfectness to be checked using specific classes of flat covers (see \Cref{lem:faithfully_flat_detect_perfect,lem:faithfully_flat_detect_perfect}).

As a corollary, we have the following.

\begin{corollary*}
    [\Cref{cor:q-perfect_locus}]
    Let $f\colon Y\to X$ be a proper morphism of Noetherian algebraic spaces. Suppose that $A_\#$ and $s\colon U\to X$ are as in \Cref{thm:quasi-perfectness_locality}. Then the locus
    \begin{displaymath}
        S:=\left\{ p \mid \text{the projection } Y\times_X \operatorname{Spec}(A_p)\to \operatorname{Spec}(A_p) \textrm{ is quasi-perfect}\right\}
    \end{displaymath}
    is an open subset of $X$ or $U$ depending on the choice of $A_\#$. 
\end{corollary*}

One can think of the open subset $S$ as the `the quasi-perfect locus of $f$.
The fact it is always open for proper morphisms appears to be new even for schemes. As an application, we show the following.

\begin{proposition*}
    Let $X$ be a reduced Noetherian algebraic space. Suppose further that $f\colon Y \to X$ is a proper surjective morphism.
    Then the regular locus of $X$ contains a nonempty open subset (i.e.\ is $J\textrm{-}0$) whenever $Y$ has this property.
\end{proposition*}

Again, to the best of our knowledge, this is new even in the case of schemes. In fact, we prove a slightly more general statement, see \Cref{prop:descending_open_in_regular_locus}. Furthermore, in \Cref{app:quasi-perfect_detect_regularity}, we record a characterization for regularity of Noetherian algebraic spaces in terms of quasi-perfectness of blowups.

Finally, we comment on the restriction to proper morphisms in \Cref{thm:quasi-perfectness_locality}. In practice, it is a mild restriction. Indeed, for a separated morphism of finite type between Noetherian schemes, quasi-perfectness already implies properness (see \cite[Proposition 0.20]{Neeman:2021}). The properness assumption is required in our arguments to employ local-to-global techniques involving bounded coherent objects (e.g.\ \cite[Theorem 4.1]{Avramov/Iyengar/Lipman:2010}).

We expect the more general results to be useful in future work on singularities of algebraic spaces. This a topic that has recently been gaining traction. For example, recently the minimal model program has been established for excellent algebraic $\mathbb{Q}$-spaces \cite{Lyu/Murayama:2022}. 

\begin{ack}
    De Deyn was supported by ERC Consolidator Grant 101001227 (MMiMMa).
    Lank was supported under the ERC Advanced Grant 101095900-TriCatApp. The authors thank Leovigildo Alonso Tarr\'{i}o, Bhargav Bhatt, Linquan Ma, Fei Peng, and Karl Schwede for discussions and/or comments. Lank would like to dedicate this work to Boz, whose joyous and warm memories left behind will be forever remembered.
\end{ack}

\section{Preliminaries}
\label{sec:prelim}

We make some recollections in this section, mainly to keep things self-contained and to fix notation and terminology.

\subsection{Generation}
\label{sec:prelim_generation}

Let $\mathcal{T}$ be a triangulated category. 
Recall that a \textbf{strictly full} subcategory is a full subcategory closed under isomorphisms. 
A triangulated subcategory of $\mathcal{T}$ is called \textbf{thick} if it is closed under direct summands (and so automatically strict, i.e.\ closed under isomorphisms). 
For any collection of objects $S$ of $\mathcal{T}$, let $\operatorname{add}(S)$ denote the smallest strictly full subcategory of $\mathcal{T}$ containing $S$ that is closed under shifts, finite coproducts and direct summands.
Following \cite{Bondal/VandenBergh:2003}, inductively, define
\begin{displaymath}
    \langle S\rangle_n :=
    \begin{cases}
        \operatorname{add}(\varnothing) & n=0, \\
        \operatorname{add}(S) & n=1, \\
        \operatorname{add}(\{ \operatorname{cone}\phi \mid \phi \in \operatorname{Hom}(\langle \mathcal{S} \rangle_{n-1}, \langle \mathcal{S} \rangle_1) \}) & n>1.
    \end{cases}
\end{displaymath}
An important feature of this construction is that $\langle S \rangle := \cup_{n\geq 0} \langle  S \rangle_n$ is the smallest thick triangulated subcategory of $\mathcal{T}$ containing $S$. 
Observe also that it is possible to rewrite the definition of $\langle \mathcal{S} \rangle_n$ using $\operatorname{Hom}(\langle \mathcal{S} \rangle_1, \langle \mathcal{S} \rangle_{n-1})$, i.e.\ there is no choice `left vs.\ right'. 

\subsection{Algebraic Spaces}
\label{sec:prelim_spaces}

Our conventions follow that of the Stacks Project \cite{StacksProject}. 
Throughout we work with algebraic spaces in the `absolute' setting, i.e.\ over the base scheme $\operatorname{Spec}(\mathbb{Z})$. Of course, everything remains valid over an arbitrary base.

Let $X$ be an algebraic space throughout this subsection.
We denote by $|X|$ the set of points of $X$, endowed with the quotient topology via any \'{e}tale surjective morphism from a scheme; see e.g.\ \cite[\href{https://stacks.math.columbia.edu/tag/03BT}{Tag 03BT}]{StacksProject}.
The algebraic spaces that are considered in this work will always be \textbf{decent} in the sense of \cite[\href{https://stacks.math.columbia.edu/tag/03I7}{Tag 03I7}]{StacksProject}.
This condition ensures that there is a good relationship between the topology of $|X|$ and properties of $X$ itself.
Note that quasi-separated algebraic spaces are automatically decent; in fact, for locally Noetherian algebraic spaces being descent is equivalent to being quasi-separated \cite[\href{https://stacks.math.columbia.edu/tag/0BB6}{Tag 0BB6}]{StacksProject}.

Defining irreducibility for algebraic spaces is a bit subtle as this notion is not \'{e}tale local.
Therefore, recall that an algebraic space $X$ is called \textbf{irreducible} when its underlying topological spaces $|X|$ is irreducible; it is not clear whether one can characterize this using coverings from schemes (e.g.\ using \'{e}tale surjective morphisms from schemes).
An algebraic spaces is \textbf{integral} if it's irreducible, reduced, and decent \cite[\href{https://stacks.math.columbia.edu/tag/0AD3}{Tag 0AD3}]{StacksProject}. 

\subsubsection{Derived categories and functors}
\label{sec:prelim_spaces_modules_derived_categories}

Let $X_{\acute{e}t}$ denote the small \'{e}tale site of $X$. 
The assignment $U \mapsto \Gamma (U, \mathcal{O}_ U)$ defines a sheaf of rings on $X_{\acute{e}t}$, denoted $\mathcal{O}_X$ and called the structure sheaf (of $X$). 
Consequently, it makes sense to speak of sheaves $\mathcal{O}_X$-modules on $X_{\acute{e}t}$; denote the (Grothendieck) Abelian category of such sheaves by $\operatorname{Mod}(X)$. 
The Abelian subcategories of quasi-coherent (resp.\ coherent) $\mathcal{O}_X$-modules are denoted by $\operatorname{Qcoh}(X)$ (resp.\ $\operatorname{Coh}(X)$). 

Next, let $D(X):=D(\operatorname{Mod}(X))$ denote the (unbounded) derived category of $\operatorname{Mod}(X)$ and let $D_{\operatorname{qc}}(X)$ resp.\ $D_{\operatorname{coh}}(X)$ be the full subcategories consisting of complexes with quasi-coherent resp.\ coherent cohomology sheaves. 
Furthermore, let $D^{\flat}_{\#}(X):=D^{\flat}(X)\cap D_{\#}(X)$ for $\flat\in \{+,-,b,\geq n,\dots\}$ and $\#\in \{\operatorname{qc},\operatorname{coh}\}$, where the superscripts have to interpreted in the usual fashion. 

Any morphisms $f\colon Y \to X$ of algebraic spaces induces a morphism of ringed topoi \cite[\href{https://stacks.math.columbia.edu/tag/03G8}{Tag 03G8}]{StacksProject}, which leads to a derived adjunction $\mathbf{L}f^\ast \colon D(X)\rightleftarrows D(Y) \colon \mathbf{R}f_\ast$ as in the case of schemes; see e.g.\ \cite[\href{https://stacks.math.columbia.edu/tag/07A6}{Tag 07A6}]{StacksProject}. 
When $f$ is quasi-compact and quasi-separated, the derived pushforward preserves objects with quasi-coherent cohomology (see e.g.\ \cite[\href{https://stacks.math.columbia.edu/tag/08FA}{Tag 08FA}]{StacksProject}) and so the adjunction restricts to $D_{\operatorname{qc}}$, i.e.\ we have an adjunction $\mathbf{L}f^\ast \colon D_{\operatorname{qc}}(X)\rightleftarrows D_{\operatorname{qc}}(Y) \colon \mathbf{R}f_\ast$. 

\subsubsection{Perfect complexes}
\label{sec:prelim_spaces_perfect}

Perfect complexes can be defined for any ringed site \cite[\href{https://stacks.math.columbia.edu/tag/08G4}{Tag 08G4}]{StacksProject}, and so, in particular, for $X$ by looking at its small \'{e}tale site.
Recall that a complex is \textbf{strictly perfect} if it is a bounded complex with each term a  direct summand of a finite free module.
A complex is \textbf{perfect} if it is locally strictly perfect. 
The subcategory of $D_{\operatorname{qc}}(X)$ consisting of perfect complexes is denoted by $\operatorname{Perf}(X)$; this is a thick subcategory. 

When $X$ is quasi-compact and quasi-separated, the perfect complexes are exactly the compact objects of $D_{\operatorname{qc}}(X)$, see \cite[\href{https://stacks.math.columbia.edu/tag/09M8}{Tag 09M8}]{StacksProject}. Recall that an object in a triangulated category with arbitrary direct sums is said to be compact if the covariant Hom out of that object commutes with direct sums. 
Moreover, in this case, $D_{\operatorname{qc}}(X)$ is compactly generated by a single perfect complex, see \cite[\href{https://stacks.math.columbia.edu/tag/0AEC}{Tag 0AEC}]{StacksProject}. 
Here, an object $P$ compactly generating means that it is compact and any non-zero object admits a non-zero morphism from some shift of $P$. Any such object is called a \textbf{compact generator}. 
It is worthwhile noting that $P$ being a compact generator for $D_{\operatorname{qc}}(X)$ is equivalent to it being a classical generator for $\operatorname{Perf}(X)$, i.e.\  $\operatorname{Perf}(X) = \langle P \rangle$.
This follows as $D_{\operatorname{qc}}(X)$ is compactly generated, see e.g.\ \cite[\href{https://stacks.math.columbia.edu/tag/09SR}{Tag 09SR}]{StacksProject}.

Lastly, for future reference, let us note the following lemma which shows that one can check perfectness flat locally; see e.g.\ \cite[Lemma 4.1]{Hall/Rydh:2017} or \cite[Lemma 2.2]{DeDeyn/Lank/ManaliRahul/Peng:2025}.
For the definition of the the fpqc topology, see \cite[\href{https://stacks.math.columbia.edu/tag/03MP}{Tag 03MP}]{StacksProject}; important examples of fpqc covers are fppf, and hence, also smooth and \'{e}tale covers.

\begin{lemma}
    \label{lem:perfect_flat_local}
    Let $X$ be an algebraic space and $E\in D_{\operatorname{qc}}(X)$.
    Then $E$ is perfect if and only if there exists an fpqc cover $\{ f_i\colon X_i\to X\}$ such that each $f^\ast_i E$ is perfect.
\end{lemma}

\subsubsection{Local rings}
\label{sec:prelim_spaces_local_rings}

Let $\overline{x}$ be a geometric point of $X$, i.e.\ $\overline{x}$ is a morphism from the spectrum of an algebraically closed field to $X$. 
For any presheaf $\mathcal{F}$ on $X_{\acute{e}t}$, one defines the stalk of $\mathcal{F}$ at $\overline{x}$ as $\mathcal{F}_{\bar x} := \mathop{\mathrm{colim}}\nolimits _{(U, \overline{u})} \mathcal{F}(U)$ where the colimit runs over the \'{e}tale neighborhoods of $\overline{x}$ in $X$.
The stalks at geometric points defined this way are exactly the stalks at points of the small \'{e}tale site of $X$ (see e.g.\ \cite[\href{https://stacks.math.columbia.edu/tag/04K6}{Tag 04K6}]{StacksProject}). 
The \textbf{\'{e}tale local ring of $X$ at $\overline{x}$} is by definition the stalk of the structure sheaf at $\overline{x}$, i.e.\ $\mathcal{O}_{X,\overline{x}}:=(\mathcal{O}_X)_{\overline{x}}$. 
This is in fact a local ring (as can be seen from the following paragraph).

To make the latter more explicit, pick an \'{e}tale neighborhood $(U,\overline{u})$ of $\overline{x}$ where $U$ a scheme. 
Then, from \cite[\href{https://stacks.math.columbia.edu/tag/04KF}{Tag 04KF}]{StacksProject}, we have isomorphisms
\begin{equation}\label{eq:concrete_etale_local_ring}
    \mathcal{O}_{X,\overline{x}} \cong \mathcal{O}_{U,\overline{u}} \cong \mathcal{O}_{U,u}^{sh},
\end{equation}
where $\mathcal{O}_{U,u}^{sh}$ is the strict Henselization of $\mathcal{O}_{U,u}$ and $u\in U$ is the point at which $\overline{u}$ is centered.
We refer to \cite[\href{https://stacks.math.columbia.edu/tag/0BSK}{Tag 0BSK}]{StacksProject} for the definition of (strict) Henselizations of local rings.
Because of \eqref{eq:concrete_etale_local_ring}, for arbitrary $x\in |X|$, by abuse of notation we sometimes denote $\mathcal{O}_{X,x}^{sh}:=\mathcal{O}_{X,\overline{x}}$ where $\overline{x}$ is any geometric point lying over $x$.

Lastly, let us briefly define the Henselian local ring at a point $x\in |X|$.
Let
\begin{displaymath}
    \mathcal{O}_{X, x}^ h := \mathop{\mathrm{colim}}\nolimits_{(U,u)} \Gamma (U, \mathcal{O}_ U)
\end{displaymath}
where now the colimit is taken over the `elementary \'{e}tale neighborhoods' (see \cite[\href{https://stacks.math.columbia.edu/tag/0BGU}{Tag 0BGU}]{StacksProject} for the definition).
As in the previous paragraph, one can make this more explicit in terms of a covering. 
Indeed, by \cite[\href{https://stacks.math.columbia.edu/tag/0EMY}{Tag 0EMY}]{StacksProject}, picking an elementary etale neighborhood $(U,u)\to (X,x)$ one has $\mathcal{O}_{X, x}^ h \cong \mathcal{O}_{U,u}^{h}$, where $\mathcal{O}_{U,u}^{h}$ is the Henselization of the local ring $\mathcal{O}_{U,u}$ of $U$ at $u$.
In particular, 
\begin{equation}\label{eq:hens_local_ring}
    (\mathcal{O}_{X, x}^ h)^{sh}=\mathcal{O}_{X, x}^{sh},
\end{equation}
see e.g.\ \cite[\href{https://stacks.math.columbia.edu/tag/0EMZ}{Tag 0EMZ}]{StacksProject}.

\subsubsection{Residue fields}
\label{sec:prelim_spaces_residue_fields}

Pick a point $x\in |X|$. 
When $X$ is decent there exists, by \cite[\href{https://stacks.math.columbia.edu/tag/03K4}{Tag 03K4}]{StacksProject}, a unique monomorphism $\operatorname{Spec}(k)\to X$ representing $x$.
The field $k$ appearing is called the \textbf{residue field of $X$ at $x$} and is denoted $\kappa(x)$. 

\subsubsection{Regularity}
\label{sec:prelim_spaces_regularity}

Recall that $X$ is \textbf{regular} if there exists an \'{e}tale surjective morphism from a regular scheme, see e.g.\ \cite[\href{https://stacks.math.columbia.edu/tag/03E5}{Tag 03E5}]{StacksProject}. 
Furthermore, a point $x\in |X|$ is \textbf{regular} if there exists an \'{e}tale morphism $f\colon U \to X$ from a scheme and a point $u\in U$ such that $f(u)=x$ and the stalk $\mathcal{O}_{U,u}$ is a regular local ring; this is well-defined by \cite[\href{https://stacks.math.columbia.edu/tag/0AH7}{Tag 0AH7}]{StacksProject}.
It is straightforward to see that an algebraic space is regular if and only if all of its points are regular.
Moreover, by \cite[\href{https://stacks.math.columbia.edu/tag/0AHA}{Tag 0AHA}]{StacksProject}, a point $x\in |X|$ is regular if and only if the \'{e}tale local ring $\mathcal{O}_{X,\overline{x}}$ is a regular local ring for any (equivalently some) geometric point $\overline{x}$ lying over $x$.

When $X$ is a Noetherian algebraic space, it follows from \cite[Theorem 3.7]{DeDeyn/Lank/ManaliRahul/Peng:2025} that regularity of $X$ is equivalent to $D^b_{\operatorname{coh}}(X) = \operatorname{Perf}(X)$. 
Furthermore, when $X$ is a separated Noetherian algebraic space of finite Krull dimension, by \cite[Theorem 5.1]{DeDeyn/Lank/ManaliRahul/Peng:2025}, regularity of $X$ is equivalent to $\operatorname{Perf}(X)$ being strongly generated, i.e.\ $\operatorname{Perf}(X)=\langle G \rangle_{n+1}$ for some $G\in \operatorname{Perf}(X)$ and $n\geq 0$.

\subsection{Morphisms}
\label{sec:prelim_morphisms}

Let $f\colon Y\to X$ now be more generally any morphism of quasi-compact quasi-separated algebraic spaces. 
Extending \cite{Lipman/Neeman:2007} to algebraic spaces we obtain the following.

\begin{definition}
    \label{def:q-perfect}
    We say $f$ is \textbf{quasi-perfect} if $\mathbf{R}f_\ast \operatorname{Perf}(Y)\subseteq \operatorname{Perf}(X)$, i.e.\ pushforward preserves compacts. 
\end{definition}

The pushforward preserving compact objects has many different characterizations. Let us just mention \cite[Theorem 3.3]{Balmer/Dellambrogio/Sanders:2016} for a nice list of characterizations when viewing the above through a tensor triangular lens. 

\begin{warning}
    A morphism of schemes being quasi-perfect should not be confused with it being a `perfect morphisms', another notion that appears in the literature. These notions coincide, for example, for a proper morphisms of Noetherian schemes \cite[Theorem 1.2]{Lipman/Neeman:2007}, but differ in general.
\end{warning}

\section{Proofs}
\label{sec:proofs}

We prove our main results. To start, we need a few lemmas that allow one to check quasi-perfectness using flat covers.

\begin{lemma}
    \label{lem:perfect_morphism_via_generation}
    A morphism $f\colon Y \to X$ of quasi-compact quasi-separated algebraic spaces is quasi-perfect if and only if $\mathbf{R}f_\ast P\in \operatorname{Perf}(X)$ for any compact generator $P$ of $Y$.
\end{lemma}

\begin{proof}
    If $f$ is quasi-perfect, it is clear that $\mathbf{R}f_\ast P\in \mathbf{R}f_\ast \operatorname{Perf}(Y)\subseteq \operatorname{Perf}(X)$.
    For, the converse direction, let $P$ be compact generator on $Y$. Our hypothesis is $\mathbf{R}f_\ast P\in \operatorname{Perf}(X)$. 
    The desired claim then follows from the following string of inclusions:
    \begin{displaymath}
        \mathbf{R}f_\ast \operatorname{Perf}(Y) = \mathbf{R}f_\ast \langle P \rangle \subseteq \langle \mathbf{R}f_\ast P \rangle \subseteq \operatorname{Perf}(X).\qedhere
    \end{displaymath}
\end{proof}

\begin{lemma}
    \label{lem:faithfully_flat_detect_perfect}
    Let $f\colon Y\to X$ be a morphism of quasi-compact quasi-separated algebraic spaces and let $s\colon U \to X$ be a flat morphism (from an algebraic space). Consider the fibered square
    \begin{equation}\label{dia:base_change_square}
        \begin{tikzcd}
            {Y\times_X U} & U \\
            Y & {X}\rlap{ .}
            \arrow["{f^\prime}", from=1-1, to=1-2]
            \arrow["{s^\prime}"', from=1-1, to=2-1]
            \arrow["s", from=1-2, to=2-2]
            \arrow["f"', from=2-1, to=2-2]
            \arrow["\lrcorner"{anchor=center, pos=0.125}, draw=none, from=1-1, to=2-2]
        \end{tikzcd}
    \end{equation}    \begin{enumerate}
        \item If $f$ is quasi-perfect and $s$ is quasi-affine, then  $f^\prime$ is quasi-perfect. 
        \item If $s$ is surjective and $f^\prime$ is quasi-perfect, then $f$ is quasi-perfect.
    \end{enumerate}
\end{lemma}

\begin{proof}
    To show the first claim, first note that when $s$ is an quasi-affine morphism so is its base change $s^\prime$. Consequently, $\mathbf{R}s^\prime_\ast$ is conservative (see e.g.\ \cite[Corollary 2.8]{Hall/Rydh:2017}) and it follows that for any compact generator $Q$ of $Y$ its pullback $(s^\prime)^\ast Q$ is a compact generator of $Y\times_X U$ (see e.g.\ \cite[\href{https://stacks.math.columbia.edu/tag/0E4R}{Tag 0E4R}]{StacksProject}).
    Thus, as $f$ is assumed quasi-perfect, $\mathbf{R}f_\ast Q$ is perfect and so also is its pullback $s^\ast \mathbf{R}f_\ast Q$. From flat base change, it follows that 
    \begin{equation}
        \label{eq:flat_bc}
        \mathbf{R}f^\prime_\ast (s^\prime)^\ast Q \cong s^\ast \mathbf{R}f_\ast Q.
    \end{equation}
    Hence, \Cref{lem:perfect_morphism_via_generation} implies $f^\prime$ is a quasi-perfect morphism.

    For the second claim, assume now that $s$ is surjective and $f^\prime$ is quasi-perfect. 
    Then $\mathbf{R}f^\prime_\ast (s^\prime)^\ast Q$ is perfect and so is $s^\ast \mathbf{R}f_\ast Q$ by \cref{eq:flat_bc}. 
    As $s$ is surjective, so is $s^\prime$, and so it is an fpqc cover as it is quasi-compact.
    Therefore, $\mathbf{R}f_\ast Q$ is perfect as this can be checked flat locally by \Cref{lem:perfect_flat_local}.
    By \Cref{lem:perfect_morphism_via_generation}, we have that $f$ is quasi-perfect, which completes the proof.
\end{proof}

There are three important cases of faithfully flat morphisms that we will apply \Cref{lem:faithfully_flat_detect_perfect}.

\begin{example}
    \label{ex:adic_completion_Henselization}
    Let $(R,\mathfrak{m})$ be a Noetherian local ring. 
    To this one can associate a few other Noetherian local rings:
    \begin{enumerate}
        \item the $\mathfrak{m}$-adic completion of $R$, denoted $\widehat{R}$,
        \item the Henselization of $R$, denoted $R^h$, and
        \item the strict Henselization of $R$, denoted $R^{sh}$.
    \end{enumerate}
    There are local ring morphisms $R\to \widehat{R}$ and $R \to R^ h \to R^{sh}$ which are flat
    and consequently faithfully flat; see e.g.\ \cite[\href{https://stacks.math.columbia.edu/tag/00MC}{Tag 00MC} and \href{https://stacks.math.columbia.edu/tag/07QM}{Tag 07QM}]{StacksProject}.
    (As an aside, there is also a flat local ring morphism $R^h\to \widehat{R}$, but we do not need it).
\end{example}

We now move to our main result.

\begin{theorem}
    \label{thm:quasi-perfectness_locality}
    Let $f\colon Y\to X$ be a proper morphism of Noetherian algebraic spaces. Suppose $A_\# \in \{\mathcal{O}_{X,\#}^h, \mathcal{O}_{X,\#}^{sh} \}$. 
    Then the following are equivalent:
    \begin{enumerate}
       \item\label{item1} $f$ is quasi-perfect,
       \item\label{item2} the base change of $f$ to $\operatorname{Spec}(A_p)$ is quasi-perfect for all $p \in |X|$.
    \end{enumerate}
    Furthermore, if $s\colon U\to X$ is an \'{e}tale surjective morphism from a quasi-compact scheme, one can take $A_\# \in \{ \mathcal{O}_{U,\#}, \widehat{\mathcal{O}}_{U,\#}\}$
    and let $p$ run over all points in $U$.
\end{theorem}

Before proving this, we recall a fact whose proof we add for lack of finding a reference.

\begin{lemma}
    \label{lem:affine_to_qs_space_is_quasi-affine_morphism}
    Let $X$ be a quasi-separated algebraic space. Then any morphism from an affine scheme to $X$ is quasi-affine.
\end{lemma}

\begin{proof}
    Suppose $U\to X$ is a map from an affine scheme. So show it is affine, by definition, we need to show, for any affine scheme $V$ and morphisms $V\to X$, that $U\times_X V $ is quasi-affine.
    As the latter is just the pullback of the morphisms $U\times V\to X\times X$ along the diagonal $X\to X\times X$, this follows from the fact the diagonal of $X$ is quasi-affine, see e.g.\ \cite[\href{https://stacks.math.columbia.edu/tag/03HK}{Tag 03HK} \& \href{https://stacks.math.columbia.edu/tag/02LR}{Tag 02LR}]{StacksProject}. \qedhere
\end{proof}

\begin{proof}[Proof of \Cref{thm:quasi-perfectness_locality}]
    $\eqref{item1}\implies\eqref{item2}$: 
    In all cases, the morphisms $\operatorname{Spec}(A_p)\to X$ is flat and quasi-affine (the latter by \Cref{lem:affine_to_qs_space_is_quasi-affine_morphism} because the source is affine and the target is quasi-separated). 
    Hence, it follows from \Cref{lem:faithfully_flat_detect_perfect} that the base change of $f$ to $A_p$ is quasi-perfect.
    Note that this direction did not require $f$ to be proper. 

    $\eqref{item2}\implies\eqref{item1}$:
    Let $s\colon U\to X$ be an \'{e}tale surjective morphism from a quasi compact, and hence Noetherian, scheme. 
    We first reduce to the case $A_p=\mathcal{O}_{U,p}$ and $p\in U$.
    So pick a point $p\in U$ and let $R:=\mathcal{O}_{U,p}$.
    Observe that $R^{sh}=\mathcal{O}_{X,\overline{x}}=:\mathcal{O}_{X,p}^{sh}$ by \cref{eq:concrete_etale_local_ring}, for $\overline{x}$ a geometric point lying over $s(p)$, and that $R^{sh}=(\mathcal{O}_{X,p}^{h})^{sh}$  by \cref{eq:hens_local_ring}.
    Consequently, as $R\to \widehat{R}$, $\mathcal{O}_{X,p}^{h}\to R^{sh}$ and $R\to R^{sh}$ are faithfully flat (see \Cref{ex:adic_completion_Henselization}), it follows from \Cref{lem:faithfully_flat_detect_perfect} that the base change to one of them is quasi-perfect if and only if the base change to any of them is quasi-perfect.
    Thus, as $g$ is surjective, and so each point $x\in |X|$ is reached by a point $u\in U$, it suffices to show:
    \begin{equation}\tag{$\star$}\label{eq:star}
        \begin{minipage}{10cm}
        {If the base change of $f$ to $\mathcal{O}_{U,p}$ is quasi-perfect for all $p \in U$, then $f$ is quasi-perfect.}
        \end{minipage}
    \end{equation}
    The proof of \eqref{eq:star} is similar to that of \Cref{lem:faithfully_flat_detect_perfect}; however, note that $\operatorname{Spec}(\mathcal{O}_{U,p})\to X$ is generally not surjective so we need to do some extra work.
    Pick a point $p\in U$ and look at the fibered square 
    \begin{displaymath}
        \begin{tikzcd}
            {Y\times_X \operatorname{Spec}(\mathcal{O}_{U,p})} & {\operatorname{Spec}(\mathcal{O}_{U,p})} \\
            Y & X
            \arrow["{f_p}", from=1-1, to=1-2]
            \arrow["{t_p}"', from=1-1, to=2-1]
            \arrow["{s_p}", from=1-2, to=2-2]
            \arrow["f"', from=2-1, to=2-2]
            \arrow["\lrcorner"{anchor=center, pos=0.125}, draw=none, from=1-1, to=2-2]
        \end{tikzcd}
    \end{displaymath}
    where $s_p$ is the natural morphism and $f_p$ and $t_p$ are the projections.

    Again, choose any compact generator $Q$ of $Y$.
    The hypothesis implies that $\mathbf{R}(f_p)_\ast t^\ast_p Q$ is perfect, and so flat base change tells us $s_p^\ast \mathbf{R}f_\ast Q=(s^\ast \mathbf{R}f_\ast Q)_p$ is perfect.
    As $U$ is Noetherian, and as $p\in U$ was arbitrary, we have that $s^\ast\mathbf{R}f_\ast Q$ is perfect (this can be checked affine locally via \cite[Theorem 4.1]{Avramov/Iyengar/Lipman:2010})--this uses the fact that $s^\ast\mathbf{R}f_\ast Q$ is bounded and coherent by virtue of $f$ being proper.
    Consequently, so is $\mathbf{R}f_\ast Q$ by \Cref{lem:perfect_flat_local} and $f$ is quasi-perfect by \Cref{lem:perfect_morphism_via_generation}.
\end{proof}

\begin{corollary}
    \label{cor:q-perfect_locus}
    Let $f\colon Y\to X$ be a proper morphism of Noetherian algebraic spaces. Suppose that $A_\#$ and $s\colon U\to X$ are as in \Cref{thm:quasi-perfectness_locality}. Then the locus
    \begin{displaymath}
        S:=\left\{ p \mid \text{the projection } Y\times_X \operatorname{Spec}(A_p)\to \operatorname{Spec}(A_p) \textrm{ is quasi-perfect}\right\}
    \end{displaymath}
    is an open subset of $X$ or $U$ depending on the choice of $A_\#$. 
\end{corollary}

\begin{proof}
    By the proof of \Cref{thm:quasi-perfectness_locality} (i.e.\ using \Cref{lem:faithfully_flat_detect_perfect}) we know that 
    \begin{displaymath}
        \begin{aligned}
            S_X:=
            &\left\{ x\in|X| \mid \text{base change to }\mathcal{O}_{X,x}^{sh}  \textrm{ is quasi-perfect}\right\}
            \\ =&\left\{ x\in|X| \mid \text{base change to }\mathcal{O}_{X,x}^{h}  \textrm{ is quasi-perfect}\right\}
        \end{aligned}
    \end{displaymath}
    and 
    \begin{displaymath}
        \begin{aligned}
            S_U:=
            &\left\{ u\in U \mid \text{base change to }\mathcal{O}_{U,u} \textrm{ is quasi-perfect}\right\}
            \\ =&\left\{ u\in U \mid \text{base change to }\widehat{\mathcal{O}}_{U,u}  \textrm{ is quasi-perfect}\right\}.
        \end{aligned}
    \end{displaymath}
    Let us start by showing that 
    \begin{equation}
        \label{eq:quasi-perfect_locus_permanence}
        s^{-1}(S_X)=S_U.
    \end{equation}
    To this end, observe by \cref{eq:concrete_etale_local_ring} that $\mathcal{O}_{X,s(u)}^{sh} = \mathcal{O}^{sh}_{U,u}$ for any $u\in U$.
    As $\mathcal{O}_{U,u}\to \mathcal{O}^{sh}_{U,u}$ is faithfully flat, quasi-perfectness of base change to either of them is equivalent by \Cref{lem:faithfully_flat_detect_perfect}.
    Therefore, 
    \begin{displaymath}
        \begin{aligned}
            s^{-1}(S_X) 
            &= \left\{ u\in U \mid \text{base change to }\mathcal{O}_{X,s(u)}^{sh} \textrm{ is quasi-perfect}\right\} 
            \\&= \left\{ u\in U \mid \text{base change to }\mathcal{O}_{U,u}^{sh} \textrm{ is quasi-perfect}\right\} 
            \\&= \left\{ u\in U \mid \text{base change to }\mathcal{O}_{U,u} \textrm{ is quasi-perfect}\right\} = S_U.
        \end{aligned}  
    \end{displaymath}  
    As a consequence of \cref{eq:quasi-perfect_locus_permanence} and that $s\colon U\to |X|$ is an open quotient map, it follows that $S_X$ is open if and only if $S_U$ is open.
    Thus, picking another cover if necessary we may assume $U$ is affine and it suffices to show $S_U$ is open to finish the proof.

    Therefore, we may and will simply assume that $X$ is an affine scheme from now on. Choose any compact generator $Q$ of $Y$. 
    As $f$ is proper, $\mathbf{R}f_\ast Q$ is bounded coherent and so by \cite[Proposition 3.5]{Letz:2021} it follows that
    \begin{displaymath}
        S^{\prime}:=\left\{ x\in X \mid (\mathbf{R}f_\ast Q)_x \in \langle \mathcal{O}_{X,x}\rangle =\operatorname{Perf}(\mathcal{O}_{X,x}) \right\}
    \end{displaymath}
    is a Zariski open subset of $X$. 
    Considering the fibered square
    \begin{displaymath}
        \begin{tikzcd}
            {Y\times_X \operatorname{Spec}(\mathcal{O}_{X,x})} & {\operatorname{Spec}(\mathcal{O}_{X,x})} \\
            Y & X
            \arrow["{f_x}", from=1-1, to=1-2]
            \arrow["{t_x}"', from=1-1, to=2-1]
            \arrow["{s_x}", from=1-2, to=2-2]
            \arrow["f"', from=2-1, to=2-2]
            \arrow["\lrcorner"{anchor=center, pos=0.125}, draw=none, from=1-1, to=2-2]
        \end{tikzcd}
    \end{displaymath}
    we see that, using flatness of $s_x$ and that it is affine so that $t_x^\ast Q$ is a compact generator for $Y\times_X \operatorname{Spec}(\mathcal{O}_{X,x})$,
    \begin{displaymath}
        \begin{aligned}
            x\in S \iff
            & \text{base change to }\mathcal{O}_{X,x} \textrm{ is quasi-perfect}\\\iff&    \mathbf{R}(f_x)_\ast t_x^\ast Q \text{ is perfect} 
            \\\iff&    (\mathbf{R}f_\ast Q)_x= s_x^\ast \mathbf{R}f_\ast Q \text{ is perfect} \iff x\in S^\prime
        \end{aligned}
    \end{displaymath}
    finishing the proof.     
\end{proof}

\begin{remark}
    One can think of the subset $S$ appearing in \Cref{cor:q-perfect_locus} as the `quasi-perfect locus of $f$'.
    It follows from the first part of the proof that in general, i.e.\ regardless of properness of $f$, one can compute it \'{e}tale locally (\cref{eq:quasi-perfect_locus_permanence}) and that the specific choice of $A_\#$ does not matter.
    Moreover, this subset is open when $f$ is proper; it would be interesting to know whether or not this holds for more general morphisms. 

    It is also worthwhile to point out that, quite generally for any morphisms $f\colon Y\to X$ (of quasi-compact quasi-separated algebraic spaces) with separated target, we have an inclusion of the regular locus $\operatorname{reg}(X)\subseteq S$.
    Indeed, if $p\in|X|$ is a regular point, its \'{e}tale local ring is regular and so the base change to it is quasi-perfect.
\end{remark}

As an application of the `quasi-perfect locus', we show it can be used to descent properties of the regular locus along proper morphisms.
Moreover, note that the result requires no separations assumptions on the target and that the assumption on $X$ in the statement is a necessary condition due to the conclusion of the proposition.

\begin{proposition}
    \label{prop:descending_open_in_regular_locus}
    Let $X$ be a Noetherian algebraic space containing at least one codimension $0$ point\footnote{See \cite[\href{https://stacks.math.columbia.edu/tag/04NA}{Tag 04NA}]{StacksProject} for the relevant definition.} whose \'{e}tale local ring is reduced (e.g.\ this is the case when $X$ is reduced).
    Suppose further that $f\colon Y \to X$ is a proper surjective morphism.
    Then the regular locus of $X$ contains a nonempty open subset (i.e.\ is $J\textrm{-}0$) whenever $Y$ has this property.
\end{proposition}

\begin{proof}
    The proof starts by a series of reductions.
    Let us denote by $\eta$ the codimension zero point whose \'{e}tale local ring is reduced. 
    \begin{itemize}
        \item As $X$ admits an open dense subscheme, which in particular contains $\eta$, see e.g.\ \cite[\href{https://stacks.math.columbia.edu/tag/06NH}{Tag 06NH}]{StacksProject}, we may assume that $X$ is a scheme.
        Moreover, by choosing any affine open neighborhood $U$ of $\eta$ and by restricting to $f^{-1}(U)\to U$, we may assume $X$ is affine and, in particular, separated.
        \item Next, let $S$ denote the open subset from \Cref{cor:q-perfect_locus}; as $\eta\in S$, since $\mathcal{O}_{X,\eta}$ is a field by assumption, $S$ is non-empty.
        It follows from \Cref{thm:quasi-perfectness_locality} that the restriction $f^{-1}(S)\to S$ is quasi-perfect. So we may assume $f$ is quasi-perfect. 
        \item Let $V$ be a non-empty open subset of $Y$ contained in the regular locus, which exists by assumption. 
        If $V\neq Y$, then $Y\setminus V$ is a properly contained closed subset of $Y$. 
        As $f$ is proper, $f(Y\setminus V)$ is closed in $X$. 
        Moreover, $f^{-1}(X\setminus f(Y\setminus V))\subset V$; so restricting to the open subset $X\setminus f(Y\setminus V) \subset X$ allows us to reduce to the setting where $f$ is a quasi-perfect proper surjective morphism from a regular scheme.
        (Note that whether or not $\eta\in X\setminus f(Y\setminus V)$ does not matter, the existence of $\eta$ was needed in the previous bullet.)
    \end{itemize}
    In order to finish the proof, is suffices to show that $X$ is regular. So, let $p\in X$ be a closed point and denote the associated closed immersion by $i\colon \mathrm{Spec}(\kappa(p)) \to X$. 
    Consider the fibered square
    \begin{displaymath}
        \begin{tikzcd}
            {Y\times_X \operatorname{Spec}(\kappa(p))} & {\operatorname{Spec}(\kappa(p))} \\
            Y & {X.}
            \arrow["{f^\prime}", from=1-1, to=1-2]
            \arrow["{i^\prime}"', from=1-1, to=2-1]
            \arrow["i", from=1-2, to=2-2]
            \arrow["f"', from=2-1, to=2-2]
            \arrow["\lrcorner"{anchor=center, pos=0.125}, draw=none, from=1-1, to=2-2]
        \end{tikzcd}
    \end{displaymath}
    Clearly, $\kappa(p)$ being a field, $\mathbf{R}f^\prime_\ast \mathcal{O}_{Y\times_X} \operatorname{Spec}(\kappa(p))$ is a (non-zero, as $f$ is surjective!) finite direct sum of shifts of $\mathcal{O}_{\operatorname{Spec}(\kappa(p))}$. 
    Thus, commutativity of the diagram shows we have $i_\ast \mathcal{O}_{\operatorname{Spec}(\kappa(p))}\in \langle \mathbf{R}f_\ast D^b_{\operatorname{coh}}(Y) \rangle_1$. 
    Regularity of $Y$ ensures that $\operatorname{Perf}(Y) = D^b_{\operatorname{coh}}(Y)$ and quasi-perfectness of $f$ tells us $\mathbf{R}f_\ast \operatorname{Perf}(Y)\subseteq \operatorname{Perf}(X)$. 
    Consequently, $i_\ast \mathcal{O}_{\operatorname{Spec}(\kappa(p))} \in \operatorname{Perf}(X)$, which ensures that $p$ is a regular point of $X$. 
    As we have shown every closed point of $X$ is regular, it follows that $X$ is regular as desired.
\end{proof}

\begin{appendix}

\section{Quasi-perfect blowups detect regularity}
\label{app:quasi-perfect_detect_regularity}

The goal of this appendix is to prove \Cref{prop:regularity_via_blowups}, which gives a relationship between regularity of a Noetherian algebraic space $X$ and having `enough' quasi-perfectness of morphisms with target $X$. In an earlier version of our work, we recorded \Cref{prop:regularity_via_blowups} as an observation. However, Bhatt privately communicated to us preliminary work of his, which is to appear later, that subsumes \Cref{prop:regularity_via_blowups}.
However, we felt it worthwhile to record this (and some of the lemmas used in its proof) in an appendix.

We start with the following lemma, which should be compared to \cite[\href{https://stacks.math.columbia.edu/tag/00OC}{Tag 00OC}]{StacksProject} and gives a `spacy' version of this. 

\begin{lemma}
    \label{lem:residue_field_perfect_gives_regular}
    Let $X$ be a decent (e.g.\ a Noetherian) algebraic space.
    Suppose $x\in |X|$ is a closed point such that $i_{x,\ast}\mathcal{O}_{\operatorname{Spec(\kappa(x))}}$ is perfect, where $i_x\colon {\operatorname{Spec}}(\kappa(x))\to X$ denotes the residue field of $X$ at $x$. 
    Then $x$ is a regular point of $X$. 
\end{lemma}

\begin{proof}
    Let $(U,u)\to (X,x)$ be the `elementary \'{e}tale neighborhood' of \cite[\href{https://stacks.math.columbia.edu/tag/0BBP}{Tag 0BBP}]{StacksProject}, i.e.\ $U$ is an affine scheme, $u$ is the only point of $U$ lying over $x$, and the induced homomorphism of residue fields $\kappa(x)\to \kappa(u)$ is an isomorphism.
    We claim that the pullback of $x$ under $f\colon U\to X$ can be identified with $u$, i.e.\ we have a cartesian square
    \begin{displaymath}
        \begin{tikzcd}[ampersand replacement=\&, sep = 1.5em]
            {\operatorname{Spec} (\kappa(u))} \& {\operatorname{Spec} (\kappa(x))} \\
            U \& X\rlap{ .}
            \arrow[equal, from=1-1, to=1-2]
            \arrow["i_u"', from=1-1, to=2-1]
            \arrow["i_x", from=1-2, to=2-2]
            \arrow["f",from=2-1, to=2-2]
            \arrow["\lrcorner"{anchor=center, pos=0.125}, draw=none, from=1-1, to=2-2]
        \end{tikzcd}
    \end{displaymath}
    Indeed, let $U_x$ denote the pullback. 
    It follows from categorical nonsense that $U_x\to U$ is a monomorphism (as $\operatorname{Spec}(\kappa(x))\to X$ is).
    Moreover, as $U_x\to \operatorname{Spec} (\kappa(x))$ is \'{e}tale, it is a disjoint union of fields.
    However, as each field corresponds to a distinct point lying over $x$  (they all go to different points as monomorphisms of schemes are injective), $U_x$ is the spectrum of a field.
    Consequently, we must have  $U_x=\operatorname{Spec} (\kappa(u))$ as $U_x\to U$ is a monomorphism.
    In particular, it also follows that $u$ is necessarily a closed point as $i_x$ is a closed immersion by e.g.\ \cite[\href{https://stacks.math.columbia.edu/tag/0H1R}{Tag 0H1R} \& \href{https://stacks.math.columbia.edu/tag/0H1U}{Tag 0H1U}]{StacksProject}.

    Next, using flat base change, one sees 
    \begin{displaymath}
        f^\ast i_{x,\ast}\mathcal{O}_{\operatorname{Spec(\kappa(x))}} = i_{u,\ast}\mathcal{O}_{\operatorname{Spec(\kappa(u))}}
    \end{displaymath}
    and so also perfect.
    Taking the stalk at $u$, one has the residue field $\kappa(u)$ is a perfect $\mathcal{O}_{U,u}$-module.
    Thus, $u$ is a regular point by \cite[\href{https://stacks.math.columbia.edu/tag/00OC}{Tag 00OC}]{StacksProject} and so $x$ is a regular point of $X$ by definition.
\end{proof}

The following lemma is the main crucial, albeit simple, observation. 

\begin{lemma}
    \label{lem:regularity_via_blowups}
    Let $X$ be Noetherian algebraic space $X$.
    If the blowup of $X$ along a closed point $p\in |X|$ induces a quasi-perfect morphism, then $p$ is regular.
\end{lemma}

\begin{proof}
    Denote by $i\colon \operatorname{Spec}(\kappa(p))\to X$ the residue field associated to our closed point $p$; this is a closed immersion, see e.g.\ \cite[\href{https://stacks.math.columbia.edu/tag/0H1R}{Tag 0H1R} \& \href{https://stacks.math.columbia.edu/tag/0H1U}{Tag 0H1U}]{StacksProject}.
    Consider the fibered square
    \begin{equation}
        \label{dia:blowup_square}
        \begin{tikzcd}[sep =2em]
           E & \operatorname{Bl}_p X\;\rlap{$=: X^\prime$} \\
           {\operatorname{Spec}(\kappa(p))} & X
           \arrow["{j}", hook, from=1-1, to=1-2]
           \arrow["{g}"', from=1-1, to=2-1]
           \arrow["f", from=1-2, to=2-2]
           \arrow["i"', hook, from=2-1, to=2-2]
           	\arrow["\lrcorner"{anchor=center, pos=0.125}, draw=none, from=1-1, to=2-2]
       \end{tikzcd}
    \end{equation}
    where $f$ is the blowup of $X$ at $p$ and $E$ is the exceptional divisor. 
    By definition of blowing up, the exceptional divisor $E$ is an effective Cartier divisor, and so it follows from the distinguished triangle
    \begin{displaymath}
        \mathcal{O}_{X^\prime} (-E) \to \mathcal{O}_{X^\prime} \to j_\ast \mathcal{O}_E \to \mathcal{O}_{X^\prime} (-E)[1].
    \end{displaymath}
    that $j_\ast \mathcal{O}_E$ is perfect. 
    Hence, as $f$ is quasi-perfect by assumption, $\mathbf{R}f_\ast j_\ast \mathcal{O}_E $ is perfect too. 
    To conclude, observe from the commutativity of \Cref{dia:blowup_square} that
    \begin{displaymath}
        \begin{aligned}
            0 &\not= \mathbf{R} f_\ast j_\ast \mathcal{O}_E 
            \\&\cong i_\ast \mathbf{R}g_\ast \mathcal{O}_E 
            \\&\cong i_\ast (\text{sum of shifts of } \mathcal{O}_{\operatorname{Spec}(\kappa(p))})
            \\&\cong \text{sum of shifts of } i_\ast \mathcal{O}_{\operatorname{Spec}(\kappa(p))}
        \end{aligned}
    \end{displaymath}
    Hence $i_\ast\mathcal{O}_{\operatorname{Spec}(\kappa(p))}$ is a summand of $\mathbf{R} f_\ast j_\ast \mathcal{O}_E \in \operatorname{Perf}(X)$ and so is also perfect. Thus, $p$ is regular by \Cref{lem:residue_field_perfect_gives_regular}, and we are done.
\end{proof}

\begin{proposition}
    \label{prop:regularity_via_blowups}
    A Noetherian algebraic space $X$ is regular if and only if the blowup of $X$ along any closed point is a quasi-perfect morphism.
\end{proposition}

\begin{proof}
    If $X$ is regular, then for any proper $f$ we have
    \begin{displaymath}
        \mathbf{R}f_\ast  \operatorname{Perf}(Y)\subseteq \mathbf{R}f_\ast D^b_{\operatorname{coh}}(Y)\subseteq D^b_{\operatorname{coh}}(X) = \operatorname{Perf}(X).
    \end{displaymath}
    Hence, any blowup is quasi-perfect as these are proper.
    The converse follows from \Cref{lem:regularity_via_blowups} as it suffices to check regularity at closed points (see e.g.\ \cite[Lemma 3.2]{DeDeyn/Lank/ManaliRahul/Peng:2025}).  
\end{proof}

Recall that a morphism of integral algebraic spaces $f\colon Y \to X$ is a \textbf{modification} (resp.\ \textbf{alteration}) if it is proper and birational, i.e.\ generically an isomorphism (resp.\ generically finite).
See e.g.\ \cite[\href{https://stacks.math.columbia.edu/tag/0ACU}{Tag 0ACU}]{StacksProject} and \cite[\href{https://stacks.math.columbia.edu/tag/0AD7}{Tag 0AD7}]{StacksProject} for some characterizations.

\begin{corollary}
    \label{cor:regularity_via_alteration}
    An integral Noetherian algebraic space $X$ is regular if and only if every alteration (resp.\ modification) of $X$ is quasi-perfect.
\end{corollary}

\begin{proof}
    This follows immediately from \Cref{lem:regularity_via_blowups}. Indeed, the blowups appearing in the proof have integral target and so are modifications by \Cref{lem:blowup_integral_spaces} below.
    Moreover, as modifications are alterations the case of alterations also follows.
\end{proof}

\begin{lemma}
    \label{lem:blowup_integral_spaces}
    The blowup $f\colon X^\prime \to X$ of an integral algebraic space along a nonzero quasi-coherent ideal sheaf is integral. 
    In particular, when the ideal is of finite type, it is a modification.
\end{lemma}

\begin{proof}
    The blowup is reduced, see e.g.\ \cite[\href{https://stacks.math.columbia.edu/tag/085W}{Tag 085W}]{StacksProject}, as the latter is an \'{e}tale local property.
    Moreover, since blowups are representable, 
    it follows that $X^\prime$ is decent (see e.g.\ \cite[\href{https://stacks.math.columbia.edu/tag/0ABT}{Tag 0ABT}]{StacksProject}). 
    
    Thus, it remains to show that $|X^\prime|$ is irreducible, which is more subtle as this notion is not \'{e}tale local. 
    Let $Z$ be the center of the blow-up. We have that $f^{-1}(Z)$ is an effective Cartier divisor on $X^\prime$ and that $f$ induces an isomorphism $f^{-1}(X\setminus Z)= X\setminus Z$; see e.g.\ \cite[\href{https://stacks.math.columbia.edu/tag/085T}{Tag 085T}]{StacksProject}. 
    When $|X|$ is irreducible so is the open subset $|X\setminus Z|$ and so $|f^{-1}(X\setminus Z)| = |X^\prime \setminus f^{-1}(Z)|$ is also.
    However, $f^{-1}(Z)$ being effective Cartier implies that the open subspace $X^\prime \setminus f^{-1}(Z)$ is scheme-theoretically dense in $X^\prime$, see \cite[\href{https://stacks.math.columbia.edu/tag/083S}{Tag 083S}]{StacksProject}. 
    In particular, the closure of $|X^\prime \setminus f^{-1}(Z)|$ (in $|X^\prime|$) is $|X^\prime|$ and so the latter is also irreducible. 

    To see that $f$ is a modification when the ideal is of finite type, note that in this case $f$ is proper by \cite[\href{https://stacks.math.columbia.edu/tag/085Z}{Tag 085Z}]{StacksProject}.
\end{proof}

To finish, using \Cref{lem:faithfully_flat_detect_perfect,lem:regularity_via_blowups}, we give an alternative proof of the fact that regularity can be checked fpqc locally. 

\begin{corollary}
    \label{cor:fpqc_and_regularity}
    Let $X$ be a Noetherian algebraic space. If there is an fpqc cover $\{f_i\colon U_i \to X\}_{i\in I}$ where each $U_i$ is regular, then $X$ is regular.
\end{corollary}

\begin{proof}
    First, as $X$ is quasi-compact, we can reduce to an fpqc cover with finitely many components (i.e. $|I|<\infty$). 
    Secondly, if needed, we can refine the cover so that each $U_i$ is an affine scheme (e.g.\ take \'{e}tale surjective morphisms from affine schemes). 
    Next, consider the induced morphism $f\colon \bigsqcup_{i\in I} U_i \to X$ which is faithfully flat by construction and moreover is quasi-affine by \Cref{lem:affine_to_qs_space_is_quasi-affine_morphism} as $U:=\bigsqcup_{i\in I} U_i$ is (isomorphic to) an affine scheme. 
    To show $X$ is regular, it suffices by \Cref{lem:regularity_via_blowups} to show that blowups along nonzero ideal sheaves are quasi-perfect.
    However, as blowups commute with flat base change, see e.g.\ \cite[\href{https://stacks.math.columbia.edu/tag/0805}{Tag 0805}]{StacksProject}, we have that the base change to $U$ is again such a blow-up. 
    Consequently, as $U$ is regular, it is quasi-perfect. 
    Thus, the original blow-up of $X$ is also quasi-perfect, by \Cref{lem:faithfully_flat_detect_perfect}, showing the claim.
\end{proof}

\begin{remark}
    In fact, when looking at fppf covers it is enough to only assume that $X$ is quasi-compact and decent in the previous by making use of some descent. 
    Indeed, it follows from the existence of an fppf cover as in the statement that $X$ is locally Noetherian (as this is fppf local and regular algebraic spaces are always locally Noetherian) and hence Noetherian.
\end{remark}

\end{appendix}

\bibliographystyle{alpha}
\bibliography{mainbib}

\end{document}